\pdfoutput=1
\documentclass[12pt,reqno]{amsart}
\usepackage{amsfonts}
\usepackage{amsbsy}
\usepackage{amssymb}
\usepackage{amsmath}
\usepackage{amsthm}
\usepackage{indentfirst}
\usepackage{supertabular}
\usepackage{algorithm}
\usepackage{enumitem}
\usepackage{listings}
\usepackage{color}
\usepackage{courier}
\usepackage{algorithm}
\usepackage{algpseudocode}
\usepackage{verbatim} 
\usepackage[mathcal]{eucal}
\usepackage{multicol}
\usepackage{float}
\usepackage{graphicx}
\usepackage{mathrsfs}
\usepackage{amsthm}
\usepackage{bm}
\usepackage{amssymb}
\usepackage{mathabx}
\usepackage{hhline}
\usepackage{textcomp}
\usepackage{amsmath,amscd}
\usepackage[all,cmtip]{xy}
\usepackage{mathtools}
\usepackage{lipsum}
\usepackage{extarrows}
\usepackage{bbm}
\usepackage{tikz-cd}
\usepackage{graphicx}
\usepackage{mathrsfs}
\usepackage{amsthm}
\usepackage{amssymb}
\usepackage{amsmath,amsthm,amsfonts}
\usepackage{mathabx}
\usepackage{hhline}
\usepackage{textcomp}
\usepackage{amsmath,amscd}
\usepackage[all,cmtip]{xy}
\usepackage{mathtools}
\usepackage[margin=1in]{geometry}
\usepackage{lipsum}
\usepackage{extarrows}
\usepackage{sansmath}

\usepackage[displaymath]{lineno}
\usepackage[singlespacing]{setspace}%
\usepackage[colorlinks=true,breaklinks=true,bookmarks=true,urlcolor=blue,
citecolor=blue,linkcolor=blue,bookmarksopen=false,draft=false]{hyperref}

\providecommand{\U}[1]{\protect\rule{.1in}{.1in}}

\newcommand{\Finsler}{\parallel \cdot \parallel^n}
\newcommand{\rr}{\mathbb{R}}
\newcommand{\ff}{\mathbb{F}}
\newcommand{\nn}{\mathbb{N}}
\newcommand{\vf}{\mathcal{F}}
\newcommand{\s}{\bm{S}}

\newcommand{\mc}{MC^k}
\newcommand{\sn}{\parallel}

\DeclareMathOperator{\dd}{d}

\DeclareMathOperator{\Img}{Img}

\newcommand{\mycmd}[1]{\mbox{\sansmath$\mathsf{#1}$}}
\DeclareMathAlphabet{\mathpzc}{OT1}{pzc}{m}{it}

\newtheorem{definition}{Definition}
\newtheorem{theorem}{Theorem}
\newtheorem{lemma}{Lemma}

\newtheorem{remark}{Remark}
\newtheorem{proposition}{Proposition}

\begin{document}

\title{Finslerian geodesics on Fr\'{e}chet manifolds}
\author{Kaveh Eftekharinasab}
\address[Kaveh Eftekharinasab]{Topology lab.  Institute of Mathematics of National Academy of Sciences of Ukraine, Kyiv, Ukraine}
\email{kaveh@imath.kiev.ua}

\author{Valentyna Petrusenko}
\address[Valentyna Petrusenko]{The Higher Mathematics Department, National Aviation University, Kyiv,  Ukraine}
\email{petrusenko76@ukr.net}

\begin{abstract}

We establish a framework, namely, nuclear bounded Fr\'{e}chet manifolds endowed with Riemann-Finsler structures to study geodesic curves  on certain infinite dimensional  manifolds such as the manifold of Riemannian metrics on a closed manifold. We prove on these manifolds   geodesics exist locally and they are  length minimizing in a sense. Moreover, we show that a curve on these manifolds is geodesic if and only if it satisfies a collection of Euler-Lagrange equations.
As an application, without much difficulty, we prove that the solution to the Ricci flow on an Einstein manifold is not geodesic. 

\end{abstract}
\maketitle
\section{Introduction}

The Riemannian geometry, including geodesics, of the manifold of all Riemannian metrics on a closed manifold which is a Fr\'{e}chet manifold was studied in~\cite{g,mi}.
In these papers the geodesic equation is described explicitly, however, in practice it would be difficult to check if a curve is geodesic by the obtained formulas.
On the other hand, geodesics of other spaces such as groups of diffeomorphisms that have the structure of  Fr\'{e}chet manifolds  were investigated
by viewing Fr\'{e}chet manifolds as inverse limits of Hilbert (ILH) manifolds, cf.~\cite{ck,jl,e}. Another recent approach to study geodesics 
on Fr\'{e}chet manifolds is by considering these manifolds as projective limits of Banach manifolds, cf.~\cite{rz1,rz2}.

The reasons for these difficulties and indirect approaches are because Fr\'{e}chet analysis and geometry are rather restrictive. As for Fr\'{e}chet spaces, there is no 
general solvability theory of differential equations and the inverse
mapping theorem does not hold in general.  Hence,
for a Riemannian Fr\'{e}chet manifold the exponential map may not exist, and even if it exists it is not necessarily a local diffeomorphism at the identity.  Another concern is that there exist only weak Riemannian metrics on these manifolds and as shown in~\cite{mf,mf1}  a curve connecting two distinct points may have the zero length. Also, a torsion-free covariant derivative
compatible with a weak Riemannian metric does not exist in general. These deficiencies inhibit the study of geodesics on these manifolds.

The purpose of this paper is to develop a new natural systematic way to study geodesics on certain Fr\'{e}chet  (bounded or $ MC^k $) manifolds including  the space of smooth sections of a fiber bundle on a closed manifold. Our approach is based on a strengthened notion of differentiability (bounded or $ MC^k $-differentiability) introduced in~\cite{ml}. The  basics of Fr\'{e}chet geometry is redeveloped  under the assumption  that transition functions between the coordinate charts possess this type of differentiability in~\cite{k1}. Such generalized manifolds seem to extend the geometry of Fr\'{e}chet manifolds: for example, 
an  inverse function theorem is obtained  for this class of differentiability \cite[Theorem 4.10]{ml}.
Also, an $ MC^k $-vector field on an $ MC^k $-Fr\'{e}chet manifold $M$ has a unique $ MC^k $-integral curve (\cite[Theorem 5.1]{k1}) and in this paper we prove that it has a  local flow too, see Theorem~\ref{t3}. Also, we prove that  this flow is $\mc$-differentiable and its domain is open in $M \times \rr$ (Lemma~\ref{3}). This result is  crucial for studying geodesics on manifolds.

To define geodesics we will apply the notion of spray as in the book of Lang~\cite{l} (cf.~\cite{om1, klin} for other approaches to geodesics on infinite dimensional manifolds). A reason for this approach is that once we have the existence of
integral curves, we can carry over important results such as the existence of exponential maps
and parallel translation from the Banach case without much difficulty, indeed we shall face many similarities with the results in Banach geometry.
We also prove that, for these generalized manifolds, exponential maps are local diffeomorphisms at the identity (Proposition~\ref{expo}).

As mentioned, since Fr\'{e}chet manifolds are weakly Riemannian, the length of a curve with distinct endpoints can be zero. On an abstract infinite dimension Fr\'{e}chet manifold $M$ there are two ways to deal with this problem: use a graded weak Riemannian structure or a Finsler structure, see~\cite{sn1}. We use a collection of weak Riemannian metrics (for a graded weak Riemannian structure) and a collection of continuous functions on  the tangent bundle $TM$ (for a Finsler structure) so that together they are strong enough to induce a
topology on the tangent spaces equivalent to the one induced from the manifold
topology. Consequently, in both cases, a curve possesses a sequence of geodesic lengths. 

Herein we will use  a Finsler structure (in the sense of Palais~\cite{pl} which is a Finsler structure in the sense of Upmeier-Neeb~\cite{neeb}) as it is slightly less technical than a graded weak Riemannian structure.
 Roughly speaking a Finsler structure on an infinite dimensional Fr\'{e}chet manifold $M$ is a collection of continuous  functions on the tangent bundle $TM$ such that their restrictions
to every tangent space is a collection of seminorms that generates the same topology as the  Fr\'{e}chet model space. In addition, they satisfy a certain local compatibility condition. We should mention that our definition of a Finsler structure differs and it is far more general than the one in the finite dimensional theory.  
As pointed out by Neeb~\cite{neeb} for infinite dimensional manifolds some crucial Finsler geometric results (such as the Gauss's lemma) are not available in general and we cannot expect to have the usual machinery of
Finsler geometry. However, in the case of nuclear bounded Fr\'{e}chet  manifolds since the topology of a model space is generated by a fundamental system of $MC^{\infty}$-Hilbertian seminorms $ \sn \cdot \sn^n =  \sqrt {\langle \cdot,\cdot \rangle _{n}}  $, in fact they give rise to a Riemann-Finsler structure, we can define appropriately the concept of orthogonality. Moreover, another crucial advantage of nuclear Fr\'{e}chet manifolds (even over Banach manifolds) is that for these 
manifolds smooth vector fields can be identified with continuous derivations in the space of smooth real-valued functions on manifolds. Using these properties for an $MC^{\infty}$- nuclear Fr\'{e}chet manifold equipped with a Riemann-Finsler structure we 
prove the existence of covariant derivatives compatible with the Riemann-Finsler structure (Proposition~\ref{key}) and the Gauss Lemma (Theorem~\ref{gl}). 

In view of the arguments above we believe that the category of $MC^{\infty}$-nuclear Fr\'{e}chet manifolds 
provide a suitable setting for studying geodesics. On these manifolds, we prove that  geodesics exist locally (Theorem~\ref{23}) and they are  length minimizing in a sense (Theorem~\ref{lenm}). Also, we prove that a curve is geodesic
if and only if it satisfies a collection of Euler-Lagrange equations (Theorem~\ref{elq}). Finally, we show easily that the solution of the Ricci flow equation on an Einstein manifold is not geodesic.

It is worth noting that this category of infinite dimensional manifolds would provide an appropriate framework for studying  configuration spaces of physical field theories. As pointed out in~\cite{jl}, these spaces lead to Fr\'{e}chet manifolds
and to discuss motions  we need paths of minimal lengths. 
\section{Bounded Fr\'{e}chet manifolds}
In this section,  we shall briefly recall the basics of bounded Fr\'{e}chet manifolds but in a self-contained way for the convenience of readers, which also allows us to  establish our notations for the rest of the paper. For more studies, we refer to~\cite{k,k1,k3,ml}.

As mentioned, we use the notion of bounded or $MC^k$-differentiability. It is based on Keller's differentiability but much stronger.
Originally, in~\cite{ml} it is called bounded differentiability  but later on the term $MC^k$-differentiability has been used equivalently. 

Let $E, F$ be Fr\'{e}chet spaces, $ U $  an open subset of $ E $ and $ \varphi:U \rightarrow F $
a continuous map. Let $CL(E,F)$ be the space of all continuous linear maps from $E$ to $F$ topologized by the compact-open topology. If the directional (G\^{a}teaux) derivatives
$$\operatorname{d}\varphi(x)h = \lim_{ t \to 0} \dfrac{\phi(x+th)-\phi(x)}{t}$$
exist for all $x \in U$ and all $ h \in E $, and  the induced map  $\operatorname{d}\varphi(x) : U \rightarrow CL(E,F)$ is continuous for all
$x \in U$, then  we say that $ \varphi $ is a Keller's differentiable map of class  $C^1$. 
The higher directional derivatives and $C^k$-maps, $k\geq2$, are defined in the obvious inductive fashion.

To define  bounded differentiability, we endow a Fr\'{e}chet space $ F$ with a translation
invariant metric $\varrho$ defining its topology, and then introduce the metric concepts which strongly depend on
the choice of $\varrho$.  We consider only metrics of the following form
$$
\varrho (x,y) = \sup_{n \in \nn} \dfrac{1}{2^n} \dfrac{\sn x-y \sn^n_F}{1+\sn x-y \sn^n_F},
$$
where $  \{ \sn \cdot \sn^n _F\}_{n \in \nn}$ is a collection of seminorms generating the  topology of $F$.

Let $(E, \sigma)$ be another Fr\'{e}chet space and let $\mathbb{L}_{\sigma,\varrho}(E,F)$ be the set of all 
linear maps $ L: E \rightarrow F $ which are (globally) Lipschitz continuous as mappings between metric spaces $E$ and $F$, that is 
\begin{equation*}
\mathpzc{Lip} (L )_{\sigma,\varrho}\, \coloneq \displaystyle \sup_{x \in E\setminus\{0\}} \dfrac{\varrho (L(x),0)}{\sigma( x,0)} < \infty,
\end{equation*}
where $\mathpzc{Lip}(L)$ is the (minimal) Lipschitz constant of $L$.

The translation invariant metric 
\begin{equation} \label{metric}  
\mathbbm{d}_{\sigma,\varrho}: \mathbb{L}_{\sigma,\varrho}(E,F) \times \mathbb{L}_{\sigma,\varrho}(E,F) \longrightarrow [0,\infty) , \,\,
(L,H) \mapsto \mathpzc{Lip}(L-H)_{\sigma,\varrho} \,,
\end{equation}
on $\mathbb{L}_{\sigma,\varrho}(E,F)$ turns it  into an Abelian topological group. We always topologize the space $\mathbb{L}_{\sigma,\varrho}(E,F)$ by the metric~\eqref{metric}.

Let $ U $ be an open subset of $ E $ and let $ \varphi:U \rightarrow F $ be
a continuous map. If $\varphi$ is Keller's differentiable, $ \operatorname{d}\varphi(x) \in \mathbb{L}_{\sigma,\varrho}(E,F) $ for all $ x \in U $ and the induced map 
$ \dd \varphi(x) : U \rightarrow \mathbb{L}_{\sigma,\varrho}(E,F)$ is continuous, then $ \varphi $ is called bounded differentiable 
or $ MC^{1} $ and we write  $\varphi^{(1)} = \varphi' $. We define  for $(k>1) $  maps of class $ MC^k$, recursively. If $ \lambda (t) $ is a  curve in a Fr\'{e}chet space,
we denote its derivative by $\lambda'$ or $ \dd \lambda(t)/ \dd t $. For product spaces, we denote by  $ \dd_i $ (in the case of curves by $\partial_i$ ) the partial derivative with respect to the $i$-th variable.

An $MC^k$-Fr\'{e}chet manifold is a Hausdorff second countable topological space modeled on a Fr\'{e}chet space with an atlas of coordinate 
charts  such that the coordinate transition functions are all
$ MC^{k} $-maps. We define $MC^k$-maps between Fr\'{e}chet manifolds as usual.

We recall the definition of nuclear manifolds as we mainly work with these manifolds.
Let $(B_1, \mid \cdot \mid_1)$ and $(B_2, \mid \cdot \mid_2)$ be Banach spaces. A linear operator $T: B_1 \to B_2$
is called nuclear or trace class if it can be written in the form
$$
T(x) = \sum_{j=1}^{\infty} \lambda_j \langle x,x_j \rangle y_j,
$$
where $\langle \cdot, \cdot \rangle$ is the duality pairing between $B_1$ and its dual $(B_1', \mid \cdot \mid'_1)$, $x_j \in B_1'$ with
$\mid x_j \mid_1' \leq 1$, $y_j \in B_2$ with $\mid y_1 \mid_2 \leq 1$, and $\lambda_j$ are complex numbers such that $\sum_j \mid \lambda_j \mid < \infty$.

If $\sn \cdot \sn_F^i$ is a seminorm on a Fr\'{e}chet space $F$, we denote by $F_i$ the Banach space given by completing $F$ using the seminorm
$\sn \cdot \sn_F^i$, there is a natural map from $F$ to $F_i$ whose kernel is $\ker \sn \cdot \sn_F^i$. A Fr\'{e}chet space
$F$ is called nuclear if for any seminorm $\sn \cdot \sn_F^i$ we can find a larger seminorm  $\sn \cdot \sn_F^j$
so that the natural induced map from $F_j$ to $F_i$ is nuclear. A nuclear Fr\'{e}chet manifold is a manifold modeled on a nuclear Fr\'{e}chet space. Each nuclear Fr\'{e}chet space admits a fundamental system of Hilbertian seminorms, see~\cite{michor}. There are no infinite dimensional Banach spaces that are nuclear. A simple example of Fr\'{e}chet nuclear space
is the space of smooth functions $C^{\infty}(U,\rr)$, $U \subset \rr^n$ is open, with the fundamental system of seminorms
$$
\sn f \sn^i = \sup_{\substack {x \in S_i, \\ \mid \alpha \mid \leq i}} \mid f^{(\alpha)}(x) \mid,
$$
where $ S_1 \subset S_2 \subset S_2 \cdots$ is  an exhaustion by open sets. 

A very important example of a Fr\'{e}chet nuclear (bounded) manifold
is  the manifold of all smooth sections of a fiber bundle (such as the manifold of Riemannian metrics) on a closed manifold. For more details on nuclear spaces we refer to~\cite{michor}.       

Let $M$ be an $\mc$-Fr\'{e}chet manifold modeled on a Fr\'{e}chet space $F$. Let
$p \in M$, tangent vectors $v \in T_pM$ are defined as equivalence classes of smooth curves passing through $p$, where the equivalency means
that curves have the same derivative at $p$. We write $T M  \coloneq   \bigcup _{p \in M}T_p M$ for the
tangent bundle of $M$. The bundle projection $ \pi : T M  \to  M$ maps elements of $T_p M$  to $p$, the tangent bundle $TM$ carries a natural vector bundle structure, see~\cite[Thorem 3.1]{k1}. 

 An important feature of an $MC^k$-Fr\'{e}chet manifold $M$ (which is not true for Fr\'{e}chet manifolds in general)
is that an $ \mc $-vector field $X: M \to TM $ has a unique integral curve. More precisely,
\begin{theorem}\cite[Theorem 5.1]{k1}\label{2}
	Let $X :M \rightarrow TM$ be a vector field of class  $ \mc $, $ k \geq 1 $. Then there exits an integral curve for $X$ at $x \in M$. Furthermore, any two such curves are equal on the intersection of their domains.
\end{theorem}
Another important feature of $MC^k$-differentiability (which is not true for Keller's differentiability) is that an $ \mc $-vector field on a Fr\'{e}chet space has an $\mc$-local flow.
\begin{theorem}\cite[Theorem 2.2]{k}\label{flow}
	Let $X$ be an $ \mc $-vector field on $ U \subset F $, $ k\geq 1 $. There exists  a real number $ a > 0 $
	such that for each $ x \in U $ there exists a unique integral curve $\ell_x(t)$ satisfying $ \ell_x(0)=x $  for all $ t \in I_a = (-a, a) $.
	Furthermore, the mapping
	$\ff : I_a \times U \rightarrow F$ given by $\ff_t(x)= \ff(t,x)= \ell_x(t) $ is of class $ \mc $.
\end{theorem}
In this paper, we define the local flow  of an $MC^k$-vector field $X : M \to TM$  and prove that it has the unique $MC^k$-flow and its domain is open in $M \times \rr$. This is indeed a critical result that allows defining exponential maps. 

A motivation for defining this class of differentiability was to obtain the following inverse function theorem:
\begin{theorem}\label{inv}\cite[Theorem 4.10]{ml}
	Let $ x_0 \in U \subset M$ be open and $\varphi: U \to N$ a $ \mc $-map, $ k\geq 2 $. If $ \varphi' (x_0) $ is an isomorphism. Then there exists $r>0$ such that $	V= \varphi( B(x_0,r)) $ is open in $ N $ and $ \varphi : B(x_0,r) \to V  $  is a diffeomorphism.
\end{theorem}
In this theorem a ball is defined with respect to a metric that induces the same manifold topology, we shall use a Finsler metric.
As a consequence of this theorem, we shall prove that exponential maps are local diffeomorphisms at the identity. 

We stress again none of the above results and the ones that we shall prove are true for  Fr\'{e}chet manifolds in general. 
Most concepts and results from finite dimensional differential geometry cannot be generalized trivially and without
restrictive approaches to Fr\'{e}chet manifolds. Apart from the concepts that depend on the finite-dimensionality, there are obstructions of intrinsic character which are mainly related to dual spaces. The dual of a Fr\'{e}chet space (non-Banachable) is never a Fr\'{e}chet space and cotangent bundles do not admit differentiable (in any sense) manifold structures, see~\cite{neeb2}. Therefore, some concepts such as the musical isomorphism and strong Riemannian metrics are not at hand. Other obstacles are of analytic nature which are caused by the lack of general solvability of differential equations and the absence of an inverse function theorem in general, therefore geometrical objects such as geodesics, exponential maps and parallel translation may not exist. In this paper we overcome the latter drawbacks by working out in the category of $\mc$-manifolds.

\section{Geodesics of sprays} 
Let $M$ be an $ \mc $-Fr\'{e}chet manifold modeled on $ F $ and let $ \pi: TM \to M$ be its tangent bundle.
Suppose $ X $ is an $ \mc $-vector field $X: M \to TM $, $ k\geq 1 $. 

Let $ U $ be open, $ x \in U \subset M$ and $I_a = (-a, a), \, a \in (0,\infty]$. A local flow of $ X $ at $ x $ is an $\mc$-function $$ \ff :U \times I_a \to M$$
such that
\begin{enumerate}
	\item for each $ x \in U, \ell_x : I_a \to M$ defined by $\ell_x(t) = \ff(x,t)$ is an integral curve
	of $X$ at $ x $,
	\item if $ \ff_t: U \to M $ is $ \ff_t(x) = \ff(x,t) $ then for $ t \in I_a $, $ \ff_t(U) $ is open and
	$ \ff_t $ is an $ \mc $-diffeomorphism onto its image.
\end{enumerate}
 For $ t+s \in I_a$ we have $ \ff_{t+s}(x) =\ell_{x}(t+s) $. But 
$ \ff_{t}(\ff_{s}(x)) = \ff_{t}(\ell_x(s)) $ is the integral curve through $ \ell_{x}(s) $, and 
$ \ell_{x}(t+s) $ is also an integral curve at $ \ell_{x}(s) $ so by Theorem~\ref{2} they coincide, and on $U$ $$\ff_{t}(\ff_{s}(x))=\ell_{x}(t+s)=\ff_{t+s}(x), $$ therefore, $ \ff_{s} \circ \ff_t = \ff_{s+t}=\ff_{t+s}= \ff_t \circ \ff_{s}$.
Since $\ell_{x}(t)$ is a curve at $x$, $ \ell_{x}(0)=x $, so $ \ff_0 $ is the identity. Moreover, 
$\ff_{t} \circ \ff_{-t} = \ff_{-t} \circ \ff_{t}$ is the identity therefore, if $$ V_t = \ff_{t}(U) \bigcap U \neq \emptyset, $$ then 
$ \ff_{t} \mid_{V_{-t}}  : V_{-t} \to V_{t} $
is a diffeomorphism and its inverse is $ \ff_{-t} \mid_{V_{t}} $.

Now we prove that an $\mc$-vector field $X : M \to TM$ has  a unique local flow.

\begin{theorem}\label{t3}
	Let $ X $ be an $ \mc $-vector field on  $ M $. For each $ x \in M $
	there exists an $\mc$-local flow of $ X $ at $ x $. Let $ \ff_1 : U_1 \times I_1 \to M $ and $ \ff_2 : U_2 \times I_2 \to M $
	be two local flows then they are equal on $ (U_1 \cap U_2)\times (I_1 \cap I_2) $.
\end{theorem} 
\begin{proof}
	(Uniqueness). For each $ u \in U_1 \bigcap U_2$ we have $\ff_1 \mid_{\{u\}\times I} = \ff_2 \mid_{\{u\}\times I}$, where $I = I_1 \bigcap I_2$.
	This follows from Theorem~\ref{2} and the definition of local flows. Thus, $\ff_1 = \ff_2$ on the set $ (U_1 \cap U_2)\times I $.
	
	(Existence). In order to prove the existence we use the local representation. Let $ (x \in U, \psi) $ be a chart and let
	$ \ff : V \times I_a \to F $ be the local flow of the local representative of $ X $ at $ \psi(x) $
	given by Theorem~\ref{flow} with
	$$
	I_{a} = (-a , a), \quad V \subset \psi (U),\quad \ff (V \times I_a) \subset \psi(U).
	$$
	Define 
	\begin{equation*}
	\begin{array}{cccc}
	\overline{\ff} : \psi^{-1}(V) \times I_a \to M \\
	(u,t) \to \psi^{-1}(\ff (\psi (u),t)). 
	\end{array}
	\end{equation*}
	Since  $ \overline{\ff} $  is continuous, there exist an open neighborhood $  W  \subset \psi^{-1}(V)$ of $x$ and $ 0 < b < a$ 
	such that $$ \overline{\ff} (W \times I_{b}) \subset \psi^{-1}(V). $$
	The restriction of $ \overline{\ff} $ to $ W \times I_{b} $ is the local flow of $ X $ at $ x $. By the construction,
	$ \overline{\ff} $ is $ \mc $.
	The first condition of the definition of local flows holds because it is true for the local representative. To prove the second condition of the definition, note that for each $ t \in I_{b}$, $ \overline{\ff}_{t}$  has an $\mc$ inverse $  \overline{\ff}_{-t}$ on 
	$\psi^{-1}(V) \bigcap \overline{\ff}_{t}(W) = \overline{\ff}_{t}(W) $.
	It follows that $\overline{ \ff_t}(W) $ is open. And, since $\overline{\ff}_t$ and $\overline{\ff}_{-t}$ are both of class $\mc$, $\overline{\ff}_t$ is a $\mc$-diffeomorphism.
\end{proof}  
It follows from Theorem~\ref{2} that the union of the domains of all integral curves of an $ \mc $-vector field $X: M \to TM (k\geq 1) $ through $x \in M$  is an open interval which we denote by $I_x = (T_x^-,T_x^+)$, where $T_x^-$ (resp. $T_x^+$) are
the sup (resp., inf ) of the times of existence of the integral curves.

 Let  $\mathcal{D}_{X} \coloneq  \bigcup_{x \in M}(\{x\} \times  I_x ) $, then we have a map $\ff : \mathcal{D}_{X} \to M$ defined on the entire 
 $\mathcal{D}_{X}$ such that $\ff(x,t)$ is the local flow of $X$ at $x$. We call this the flow determined by $X$, and we call $\mathcal{D}_{X}$
 the domain of the flow. We prove that the sets
\begin{equation*}
M_t = \{x \in M \mid (x,t) \in \mathcal{D}_{X} \}
\end{equation*}
are open subsets of $ M $.

\begin{lemma}\label{3}
	The domain $\mathcal{D}_{X}$ is open in $M \times \rr$. Moreover, the set $ M_t $ is open in $M$ for each $t \in \rr$.
\end{lemma}
\begin{proof}
	We follow the idea of~\cite[Theorem 2.6]{l}.	
	Let $ x \in M $ and let $ J_x \subseteq I_x $  be the set of points for which $ U \times (t-a,t+a)  \subseteq \mathcal{D}(X) $ for some positive number $ a $ and an open neighborhood $ x \in U $, and such that the restriction of the flow $ \ff $
	of $ X $ to this product is an $ MC^{k}$-map. Then, the interval $ J_x $ is open
	in $ I_x $ and it contains zero by Theorem~\ref{t3}. 
	
	We show that $ J_x $ is closed in $ I_x $  too. 
	Let $ s $ belong to its closure $ \overline{J_x} $.  By Theorem~\ref{t3} we can find a neighborhood $ V $ for 
	$ \ff (x,s) $ such that there is a unique  $\mc$- local flow $$ \mathbb{E} :  V \times I_b \to M, $$ for some positive number $ b $
	and $ \mathbb{E}(v,0) = v $ for all $ v \in V $. 
	
	Let a neighborhood $ \ff(x,s) \in V_1 \subseteq V $ be small
	enough. By the definition of $ J_x $, there exist $ t_1 \in J_x $ close enough to $ s $ and a small number $ \bar{a} $ and a
	small enough neighborhood $ x \in W $ such that on this product $\ff$ is $\mc$ and $$   \ff (W \times(t_1-\bar{a}, t_1+\bar{a}))\subseteq V_1 .$$
	
	Define $$ \mathcal{F} (w,t) = \mathbb{E} (\ff(w,t_1),t-t_1) $$ for $ w \in W $ and $ t $ belongs to the translation
	of $ I_b $ by $ t_1$, $  I_b+t_1  $. Then $$ \mathcal{F}(w,t_1) = \mathbb{E}(\ff(w,t_1),0) = \ff(w,t_1), $$ and by the chain rule (\cite[Lemma B.1 (f)]{hg}
	\begin{align}
	\dfrac{\dd}{\dd t}\mathcal{F}(w,t) &= \dd_2 \mathbb{F}(w,t) \circ \dd_2 \mathbb{E} (\ff(w,t_1),t-t_1,) \nonumber \\
	&=X (\mathbb{E}(\ff(w,t_1),t-t_1) = X (\mathcal{F}(w,t)) \nonumber.
	\end{align}
	Therefore, both $ \ff (x,t) $ and $ \mathcal{F}(x,t) $ are integral curves of $ X $ with $$ \ff(x,t_1) = \mathcal{F}(x,t_1) .$$
	Thus, they coincide on the intersection of their domains and $ \mathcal{F}(t,x) $ is an extension of
	$\ff(x,t)$ to a bigger interval containing $ s $, therefore, $ J_x $ is closed in $ I_x $ and consequently $ J_x =I_x $. 
	Since $ \ff $ is $ \mc $  on $ W \times (t_1-\bar{a},t_1+\bar{a})$
	it follows that $ \mathcal{F} $ is $ \mc $ on $ W \times (I+t_1)  $. Whence,  $\mathcal{D}(X) $ is open in $  M \times \rr $
	and consequently $ M_t $ is open in $ M $, and $ \ff $ is of class $ MC^{k} $ on the whole domain $ \mathcal{D}(X) $. 
\end{proof}
 The double tangent bundle $T(TM)$ over $TM$ has two vector bundle structure, one determined by the natural projection $\pi_{TM}:T(TM) \to TM$ (see~\cite[Theorem 3.1]{k1}) and the other by the tangent map $\pi_{*} =T\pi  : T(TM) \to TM$. Indeed, the tangent map is a vector bundle morphism (the arguments for Banach manifolds are valid for $ M $, see~\cite[Page 52]{l}).

Suppose $ M $ is of class $ \mc, \, k\geq 3 $.
Let $\alpha: I \to M$ be an $MC^l (l\geq 2)$-curve, a lift of $\alpha$ into $TM$ is a curve $\hat{\alpha} : I \to TM$ such that $\pi \hat{\alpha} = \alpha$. The derivative
$\alpha' : I \to TM$ is called the canonical lift. A second order vector field over $M$ is a vector field $\vf : TM \to T(TM)$ such that $$\pi_{*} \circ \vf =Id_{TM}.$$
An integral curve $\imath : I \to TM$ of $\vf$ is equal to the canonical lift of $\pi \imath$, that is $$(\pi \imath)'= \imath.$$

A geodesic with respect to $\vf$ is a curve $\mathbf{g}: I \to M$ such that its derivative $\mathbf{g}' : I \to TM$ is an  integral curve of $\vf$, that is $\mathbf{g}'' =\vf (\mathbf{g}')$. 

Let $ s\neq 0 \in \rr$ be fixed, define the mapping 
\begin{equation*}
\begin{array}{cccc}
L_s: TM \to TM \\ v \mapsto sv.
\end{array}
\end{equation*}
A second order vector filed $\s : TM \to T(TM)$ is said to be spray if 
\begin{enumerate}
	\item $\pi_{*}\s(v)=v$,
	\item $\s(sv) = (L_s)_{*}(s\s(v))$ for
	all $ s \in \rr $ and $ v \in TM $.
\end{enumerate}

If a manifold admits a partition of unity, then there exists a spray over $ M $, cf.~\cite[Theorem 3.1]{l}. Let $ U \times F $ be a chart for $ TM $ and let $ \phi: U \times F \to F \times F $
with $ \phi = (\phi_1,\phi_2) $ be a map. By repeating the arguments of ~\cite[Proposition 3.2]{l} and the remarks after it we obtain that $ \phi $ represents a spray $ \s $ if and only if $\phi_1(x,v)=v$ and
$$
\phi_2 (x,v) = \dfrac{1}{2}\dd_2^2\phi_2(x,0)(v,v).
$$
Thus, at $ x \in U $ in the chart the spray is determined by a symmetric bilinear map 
\begin{equation}\label{sr}
\mycmd{S}(x)= \dfrac{1}{2}\dd_{2}^2\phi_2(x,0).
\end{equation} 
Let $\s$ be a spray over $M$. If $\imath : I \to TM$ is an integral
curve of $\s$, then  $\imath$ is the canonical lift of the curve
$\ell \coloneq \pi \circ \imath : I \to M $, that is, $\imath = \ell'$. Thus, $ \ell $ is a geodesic of $\s$ because
$\ell'' = \imath' = \s \circ \imath = \s \circ \ell'$.
If $\ell : I \to M$ is a geodesic of $\s$, then its canonical lift $\imath = \ell'$ is
an integral curve of $\s$.
Therefore, a curve $\ell: I \to M$ is
a geodesic of $\s$ if, and only if, $\ell'$ is an integral curve of $\s$.
\begin{lemma}
	Let $\s$ be a spray of class $\mc$, $ k\geq 2 $, over $M$. If $ x \in M $ and $v$ is a tangent vector in $T_xM$, then there exists the unique integral curve $\imath: I \to TM$ of $\s$ such that $\imath(0)=v$.
\end{lemma}
\begin{proof}
	The spray $ \s $ is a vector field on $ TM $ so by Theorem~\ref{2} it has a unique  integral curve $ \imath : I \to TM $ such that $ \imath (0)=v $. The integral curve $ \imath $ is the canonical lift of the  geodesic $ \ell = \pi \circ \imath $ and $ \ell'(0)= \imath(0)=v $.
	
	If $ \ell_1 : J  \to M$ is another  geodesic with $ \ell_1'(0)=v $, then $ \imath_1 = \ell_1' $
	is also an integral curve of $ \s $ such that $ \imath_1(0)=v $ and so $ \imath_1 = \imath $.
\end{proof}
Let $ v \in TM $. By the previous lemma there exists a unique  integral curve  
$\imath_v: I_v \to TM$ of $\s$ such that $\imath_v(0)=v$. For  $ v \in TM $ we have the following result:
\begin{lemma}\label{ex}
	Let $ s,t \in \rr $, then for a fixed $ s$ and all $ t $ such $ st \in I_v $ we have $$ \ell_{sv}(t) =s \ell_{v} (st).$$	
\end{lemma}
\begin{proof}
	Let a fixed $ s $ be given and $ t\in \rr $ be such that $ st \in I_v $, then the curve $ \ell_{v}(st) $ is defined
	and
	\begin{equation}
	\dfrac{\dd}{\dd t}(s \ell_{v}(st)) = (L_s)_{*}s\ell_{v}'(st) = (L_s)_{*}s\s (\ell_{v}(st))=\s (s\ell_{v}(st)).
	\end{equation}
	Therefore, the curve $ s\ell_{v}(st) $ is a unique integral curve of $ \s $ such that $ s\ell_{v}(0) =sv $ and
	the uniqueness of the integral curve implies that $ \ell_{sv}(t) =s \ell_{v} (st)$.
\end{proof}
Let $ \s $ be a spray on $ M $ of class $ \mc $, $ k\geq2 $. Let $\ell_v  $ be the integral curve
of $ \s $ with the initial condition $ v \in TM $. 
Let $$ \mathcal{D} \coloneq \{ v\in TM \mid \ell_v  \textrm{ is defined at least on}\, [0,1]\} . $$
By Lemma~\ref{3}, $ \mathcal{D} $ is an open set in $ TM $ and $ v \mapsto \ell_v(1) $ is an $ \mc $-map.

We define the exponential map by 
\begin{equation}
\begin{array}{cccc}
\exp :\mathcal{D} \to M \\ \exp(v) = \pi \ell_v(1).
\end{array}
\end{equation}
We denote by $ \exp_x : T_x M \mapsto M $ the restriction to the tangent space $ T_xM $ for $ x \in M $.
By the definition of spray for $ s = 0 $  at the zero vector $0_x$ in $ T_xM $ we have $ \s (0_x) =0 $ so
$ \exp (0_x)=x $.

\begin{proposition}\label{expo}
	Let $ M $ be an $ \mc $-Fr\'{e}chet manifold, $k\geq 3$, and let $ \exp: \mathcal{D} \to M $ be the exponential map.
	Then for each $ x \in M $, $ \exp_{x}: T_xM \to M $ is a local diffeomorphism at $ 0_x $.	
\end{proposition}
\begin{proof}
	Let $ v \in T_xM $ and let $I_v$ be an interval containing zero.  Consider the parameterized straight line 
	\begin{equation*}
	\begin{array}{cccc}
	\imath_v : I_v \to TM  \\ t \mapsto tv.
	\end{array}
	\end{equation*}
	In view of Lemma~\ref{ex} for $ s=1 $ we obtain $ \exp(tv) = \pi \ell_{tv}(1)=\pi \ell_{v}(t) $. Thereby,
	$$ (\exp(tv))' = (\pi \ell_{v}(t))' =  \ell_{v}(t),$$
	but 
	$$
	(\exp(tv))' = \exp_{*}i_v'(t). 
	$$
	Then, by evaluating at $ t= 0 $ we get $ (\exp_{*})(0_x) = Id$. Thus, the map $(\exp_{*})(0_x)$ is a linear isomorphism  and hence the inverse mapping theorem, Theorem ~\ref{inv}, implies that 
	$ \exp_{x}$ is a local diffeomorphism at $ 0_x $.
\end{proof}
Given a point $ x \in M $, by the preceding proposition and the inverse mapping theorem there exists a star-shaped open neighborhood $ \mathcal{W} $ of $ 0_x \in T_xM $ and an open neighborhood $ \mathcal{U} $  of $ x $ such that $ \exp_x : \mathcal{W} \to \mathcal{U}$ is a diffeomorphism.
The pair $ (\mathcal{U}, \mathcal{W} )$ is called a normal neighborhood of $ x $ in $ M $. 

We should note that our notion of a normal neighborhood differs from the normal coordinates in the classical sense. We shall give  normal neighborhoods in terms of the so-called injectivity radius later on.

\begin{proposition}\label{geod}
	Let $ x \in M $, $ v \in T_xM $ and $ \alpha_{v}(t) = \exp_x(tv) $. Then $ \alpha_{v}(t) $
	is a geodesic. Conversely, if $ \alpha: I \to M $ is an $ MC^2 $ geodesic with $ \alpha(0)=x $ and $ \alpha'(0)=v $.
	Then $ \alpha(t) = \exp_x(tv) $.  
\end{proposition}
\begin{proof}
	The proof is standard so we omit it.
\end{proof}
\section{Covariant derivatives} 
In this section, we work in the category of $MC^{\infty}$-Fr\'{e}chet manifolds. 

Let $ M $ be an $MC^{\infty}$-Fr\'{e}chet manifold modeled on a Fr\'{e}chet space $F $ and $ \mathcal{E}(M) $ the set of smooth real-valued maps on $ M $. Let $  \mathcal{V}(M) = MC^{\infty}(M \to TM) $ be the set of all $MC^{\infty}$-vector fields  and $X,Y \in \mathcal{V}(M)$. 

The Lie derivative 
of $ \varphi \in \mathcal{E}(M) $ with respect to a vector field $ X $ with the flow $ \ff $ is defined as usual by
\begin{equation*}
\begin{array}{cccc}
\mathcal{L}_X\varphi(x) = \lim_{t \to 0} \dfrac{\varphi(\ff(x,t))-\varphi(x)}{t}.
\end{array}	
\end{equation*}
It is easily seen that $ \mathcal{L}_X\varphi = X(\varphi) $  belongs to $ \mathcal{E}(M) $. 

Let $ (U_i,\psi_i) $ be an atlas of $ M $. We endow $  \mathcal{E}(\psi_i(U_i)) \coloneq \mathcal{E}(\psi_i(U_i), \rr) $ with the topology of uniform convergence  on compact sets, for the function and all its derivatives, that is, the weakest topology for which the maps 
$$
\varphi \to \dd^n \varphi \in C( \psi_i(U_i)\times F^n,\rr)
$$ 
are continuous, where $C(\psi_i(U_i)\times F^n, \rr)$ is the space of  continuous linear functions endowed with the compact-open topology.

Then, we equip $\mathcal{E}(M) $ with the weakest topology for which the maps $$ \varphi \mapsto \varphi \circ \psi_i^{-1} $$
from $\mathcal{E}(M) $ to $  \mathcal{E}(\psi_i(U_i))$ are continuous. The
topology of $\mathcal{E}(M) $ can also be viewed as the weakest topology for which the restrictions
$\mathcal{E}(M)  \to  \mathcal{E}(U_i) $) are continuous.
This topology is independent of the choice of atlas, see~\cite[Lemma 2]{n}.

 We identify $ \mathcal{V}(U_i) $ with $ \mathcal{E}(U_i, U_i \times F) $, then
we similarly define the topology of $ \mathcal{V}(M) $ to be the weakest topology for which the restrictions $ \mathcal{V}(M) \to  \mathcal{V}(U_i) $
are continuous, see~\cite[Page 280]{n}. 

The following theorem is proved for Fr\'{e}chet manifolds in~\cite{n} for smoothness in the sense of Keller. Careful  analysis of the proof of the theorem
shows that it has a topological nature and since $\mc$-differentiable maps are Keller's differentiable so the theorem is also valid for the subcategory of $\mc$-Fr\'{e}chet manifolds.
\begin{theorem}\cite[Theorem]{n} \label{imp}
	Let $ M $ be a regular smooth nuclear Fr\'{e}chet manifold. Then the map $ X \mapsto \mathcal{L}_X $
	is a linear topological isomorphism of the space $ \mathcal{V}(M) $ onto the space of continuous derivations in $ \mathcal{E}(M) $.  
\end{theorem} 
In~\cite{k3} a covariant derivative for $\mc$-Fr\'{e}chet manifolds is defined by means of a connection map and Christoffel symbols.
However, that definition is not consistent with our context here as we need that a covariant derivative comes from a spray. 
Herein, we adapt the definition of a covariant derivative in the sense of Lang~\cite{l}.

If $\varphi \in \mathcal{E}(M)$ and 
$X \in \mathcal{V}(M)$, then we obtain an $MC^{\infty}$-function on $M$ via 
$$ X.\varphi := \dd \varphi \circ X : M  \to \rr. $$
For $X, Y \in \mathcal{V}(M)$, there exists a unique a vector
field $ [X,Y] \in \mathcal{V}(M)$ determined by the
property that on each open subset $U \subset M$ we have 
$$ [X,Y].\varphi = X.(Y.\varphi) - Y.(X.\varphi)  $$
for all $\varphi \in MC^\infty(U,\rr)$, see~\cite{k}.  
If we again denote the local representatives of $ X,Y $ in an open set $ U \subset F $ by themselves, then 
the local representation of  $[X,Y]$ is given by
\begin{equation*}
[X,Y](x) = X'(x)Y(x)-Y'(x)X(x).
\end{equation*}
By the definition we see that $ [X,Y] $ is bilinear in both arguments and $$ [X,Y]=-[Y,X],$$ and
$$
[X,[Y,Z]]=[[X,Y],Z]+[Y,[X,Z]].
$$
\begin{definition}\label{cov}
	Let $ \pi : TM \to M $ be the tangent bundle. A covariant derivative $ \nabla $
	is an $\rr$-bilinear map 
	\begin{equation*}
	\begin{array}{cccc}
	\nabla : \mathcal{V}(M) \times \mathcal{V}(M) \to \mathcal{V}(M) \\
	(X,Y) \to \nabla_XY
	\end{array}
	\end{equation*} 
	such that for all $ \varphi \in \mathcal{E}(M)$ and $ X,Y \in \mathcal{V}(M) $ the following hold 
	\begin{enumerate}
		\item $ \nabla_{\varphi X}Y = \varphi \nabla_{ X}Y$,
		\item $\nabla_{ X}(\varphi Y) = (\mathcal{L}_{X}\varphi)Y + \varphi \nabla_{ X}Y$,
		\item $\nabla_XY - \nabla_YX = [X,Y]$.
	\end{enumerate}
\end{definition}

In a chart $ U $ we index objects by $ U $ to show their representatives. Let $ \s $ be a spray on $ M $ 
and let $ \mycmd{S}_U(x) $ as in~\eqref{sr} be the symmetric function  associated with $ \s $ in $ U $. In a chart
$ U $, define
\begin{equation}\label{cdc}
(\nabla_XY)_U(x) = Y'_U(x)X_U(x)-\mycmd{S}_U(x)(X_U(x),Y_U(x)).
\end{equation}
It is a covariant derivative over $ U $ and it does not depend on the choice of a local chart, the proof is straightforward and similar to \cite[Theorem 2.1]{l}.

Now, we define a covariant derivative along a curve. Let $ I $ be an open interval in $ \rr $,  $ \lambda: I \to M$
a curve  and  $ \widehat{\lambda}: I \to TM$ its lift. Let $ \textrm{Lift} (\lambda) $ be the vector space of lifts of 
$ \lambda $. In a chart $ U $, define the operator 
\begin{equation}\label{geo}
\begin{array}{cccc}
\nabla_{\lambda'}: \textrm{Lift} (\lambda) \to \textrm{Lift} (\lambda) \\
\big (\nabla_{\lambda'}\gamma \big)_{U}(t) = \gamma'_U(t)-\mycmd{S}_U(\lambda(t)) \big (\lambda'_U(t),\gamma_U(t) \big ).
\end{array}
\end{equation}
This defines a covariant derivative
and it does not depend on the choice of a local chart and for a mapping $ \varphi $ it satisfies the derivation property
$$
\big (\nabla_{\lambda'}(\varphi \gamma) \big)(t) = \varphi'(t)(\nabla_{\lambda'}(\gamma))(t)+\varphi(t)(\nabla_{\lambda'}\gamma)(t),
$$
the proof is standard so we omit it, cf. ~\cite[Theorem 3.1]{l}.
Let $ X $ be a vector field such that $ \gamma(t)= X(\lambda(t)) $ for $ t \in I $ and
let $ Y $ be a vector field such that $ Y(\lambda(t_0)) = \lambda'(t_0) $ for some $ t_0 \in I $.
Then by the chain rule and~\eqref{geo} we have
\begin{equation*}
(\nabla_{\lambda'} \gamma)(t_0) = (\nabla_YX) \big (\lambda(t_0) \big).
\end{equation*} 
Let $ J $ be an open interval in $ \rr $,  $ \mu: J \to M$ a $ \mc$-curve $(k\geq 2)$, and  $ \gamma: J \to TM$ a lift of $\mu$. 
We say that $ \gamma $ is $ \mu $-parallel if $ \nabla_{\mu'}\gamma =0 $. By~\eqref{geo} in a local
chart we have $$\gamma'_U(t) = \mycmd{S}_U(\mu(t)) \big (\mu'_U(t),\gamma_U(t) \big ),$$
and hence $ \mu$ is a geodesic for the spray $ \s $ if and only if $ \nabla_{\mu'}{\mu'} =0$.

\section{Finsler structures and geodesics}
As mentioned on  a Fr\'{e}chet manifold there exist only weak Riemannian metrics with unsatisfactory properties.
Thus, we use a graded weak Riemannian structure or a Finsler structure instead. The idea  behind a graded weak Riemannian metric structure
is considering not one weak metric but a collection of weak metrics such that the family of induced seminorms generates the same topology
as the Fr\'{e}chet model space. Nevertheless, this is not enough to produce a strong enough topology on the tangent spaces, in addition, the induced seminorms need to satisfy an estimation of a tame type. 

In the finite dimensional theory of Finsler manifolds,  a Finsler structure is a function $F: TM \to \rr^{+}$ which is smooth 
on the complement of the zero section and positively homogeneous and strongly convex on each tangent space.  This definition is too restrictive and insufficient for infinite dimensional Fr\'{e}chet manifolds. 
By contrast, in the infinite dimensional theory there are two definitions of Finsler structures: one in the sense of Palais and another in the sense of Upmeier-Neeb which are   different by their local compatibility conditions. Roughly speaking a Finsler structure is a collection of continuous functions on the tangent
bundle  such that their restrictions
to every tangent space is a collection of seminorms that generates the same topology as the  Fr\'{e}chet model space. In addition, this family of seminorms
needs to satisfy a certain local compatibility condition. The infinite dimensional theory of Finsler manifolds is much less general than the finite dimensional theory and analogue notions and results may not be available.  

In this paper we use the definition of a Finslear structure  in the sense of Palais~\cite{pl}.
\begin{definition}\label{defni}\cite[Definition 4.2]{k3}
	Let $F$ be a Fr\'{e}chet space $T$  a topological space, and $V = T \times F$ the trivial bundle with fiber $F$ over $T$. A Finsler
	structure for $V$ is a collection of continuous functions $ \Finsler: V \rightarrow \mathbb{R}^+$, $n \in \nn$, such that 
	\begin{enumerate}
		\item For $b \in T$ fixed, $\parallel (b,f)\parallel^n = \parallel f \parallel_b^n$ is a collection of seminorms 
		on $F$ which gives the topology of $F$.
		\item Given $ k >1$ and $t_0 \in T$, there exists a neighborhood $\mathcal{U}$ of $t_0$  such that 
		\begin{equation} \label{ine}
		\dfrac{1}{k}\parallel f \parallel^n_{t_0} \,\leqq \,\parallel f \parallel^n_{u}\, \leqq k \parallel f \parallel^n_{t_0}
		\end{equation}
		for all $u \in \mathcal{U}$, $n \in \nn$, $f \in F$.
	\end{enumerate}
\end{definition}
Suppose $M$ is a bounded Fr\'{e}chet manifold modeled on $F$. 
Let $\pi_M : TM \rightarrow M$ be the tangent bundle and let $\Finsler: TM \rightarrow \mathbb{R}^+$ be a collection of functions, $n \in \nn$. We say 
$\{\Finsler\}_{n \in \nn} $ is a Finsler structure for 
$TM$ if for a given $x \in M$, there exists a bundle chart $\psi : U \times F \simeq TM\mid_U$ with $x \in U$  
such that
$$\{\Finsler \circ \, \psi^{-1}\}_{n \in \nn} $$
 is a Finsler structure for $U \times F$.

A bounded Fr\'{e}chet Finsler manifold is a bounded Fr\'{e}chet manifold together with a Finsler structure on its tangent bundle.
If $\{ \Finsler\}_{n \in \nn}$ is a Finsler structure for $M$, then eventually we can obtain a graded Finsler structure, $(\Finsler)_{n \in \nn}$, for $M$, that is $\parallel \cdot \parallel^{i} \leqq \parallel \cdot \parallel^{i+1}$ for all $i$.

We define the length of an $MC^1$-curve $\gamma : [a,b] \rightarrow M$ with respect to the $n$-th component by 
\begin{equation*}
L_n(\gamma) = \int_a^b \parallel \gamma'(t) \parallel_{\gamma(t)}^n\, dt.
\end{equation*}
The length of a  piecewise path  with respect to the $n$-th component is the sum over the curves constituting the path. So, a curve $\gamma$
possesses a sequence of geodesic lengths $L_n(\gamma)$. By abuse of language, we say that   
the length of a curve $\gamma$ is minimal  if for all other such curves $\lambda$, we have $L_n(\gamma) \leqq L_n(\lambda)$ for all $n$. On each connected component of $M$, the distance is defined by
\begin{equation*} 
\rho_n (x,y) = \inf_{\gamma} L_n(\gamma),
\end{equation*}
where infimum is taken over all continuous piecewise $MC^1$-curve connecting $x$ to $y$. Thus,
we obtain an increasing sequence of metrics $\rho_n(x,y)$ and define the distance $\rho$ by
\begin{equation}\label{finmetric}
\rho (x,y) = \displaystyle \sum_{n = 1}^{\infty} \dfrac{1}{2^n} \cdot \dfrac {\rho_n(x,y)}{ 1+ \rho_n(x,y)}.
\end{equation}
\begin{theorem}~\cite[Theorem 4.6]{k3}
	Suppose $M$ is connected and endowed with a Finsler structure $(\Finsler)_{n \in \nn}$. 
	Then the distance $\rho$ defined by~\eqref{finmetric} is a metric for $M$, called the Finsler metric. Furthermore, the topology induced by this metric coincides with the original topology of $M$. 
\end{theorem}
If a manifold admits a partition of unity, then it possesses a Finsler structure, in particular, nuclear Fr\'{e}chet manifolds can be equipped with Finsler
structures, cf.~\cite[Proposition 4.4]{k3}.
\begin{definition}
	Let $ F $ be a Fr\'{e}chet space. A continuous function $ \mid \cdot \mid : F \to \rr^{+} $ is said to be the
	pre-Finsler norm on $ F $ if 
	\begin{enumerate}
		\item it is positive homogeneous of order 1, 
		\item it is sub-additive.
	\end{enumerate}
\end{definition}
\begin{definition}\label{f1}
	Let $ (F, \mid \cdot \mid) $ be a pre-Finsler space, a function $ \ll \cdot,\cdot \gg : F \times F \to \rr$
	is said to be the Finslerian  product if 
	\begin{enumerate}
		\item it is positive homogeneous of order 1 in its first argument, 
		\item it is linear in its second variable.
	\end{enumerate}  
\end{definition}
We say that a vector $ v \in F$ is F-orthogonal to $ u \in F $ if 
$ \ll u,v\gg_n = 0, \forall n \in \nn $.

Let $ M $ be a nuclear Fr\'{e}chet manifold of class $MC^{\infty}$ with a Finsler structure $(\Finsler)_{n \in \nn}  $.
Let $ x \in M $ and $ u, v \in T_xM $. The tangent space $ T_xM $  admits  semi-inner products by Hilbertian seminorms 
$\sn v \sn^n_x = \sqrt{\langle v,v\rangle_{n,x}}$. 
We define the Finslerin products on $ T_xM $ simply by 
\begin{equation}\label{semri}
\ll u,v \gg_{n,x} = \langle u, v \rangle_{n,x}, \forall n \in \nn.
\end{equation}
For the sake of brevity we write $ \ll u,v \gg_n $ instead of $ \ll u,v \gg_{n,x}$ where the confusion may not occur.

In local charts, mappings $\ll \cdot, \cdot \gg_n $ are linear so smooth in the sense of
Keller. Also, in local charts, the Cauchy-Schwartz inequality yields that they are globally Lipschitz and so of class $MC^{\infty}$ by Lemma B.1(a)~\cite{hg}.
\begin{remark}
	For nuclear F\'{e}chet manifolds a Finsler structure $(\Finsler)_{n \in \nn}  $ in fact is given by semi-inner products and the products~\eqref{semri} are Riemannian. Therefore, on each tangent
	space the topology is induced by a family of weak Riemannian metrics that satisfy the Finsler condition. In such a case, we call $(\Finsler)_{n \in \nn}  $ a Riemann-Finsler structure. It is to be observed that we cannot use an arbitrary collection of weak metrics they need to satisfy
	the Finsler condition (Definition~\eqref{defni}); this justifies the terminology  ``Riemann-Finsler structure".  
\end{remark}
If $ X,Y $ are vector fields, then $ \ll X,Y \gg_n $ is a function on $ M $
with the value $ \ll X(x),Y(x) \gg_n $ at a point $ x \in M $.
\begin{proposition}\label{key}
	Let $ M $ be an $MC^{\infty}$-nuclear  Fr\'{e}chet manifold with a Riemann-Finsler structure $ ( \Finsler )_{n \in \nn} $.
	Then for each $ n \in \nn $ there exists a unique covariant derivative $\nabla^n$ such that
	\begin{equation}\label{fc}
	\nabla^n_Z \ll X,Y \gg_n = \ll \nabla^n_Z X, Y\gg_n + \ll X,\nabla^n_Z Y\gg_n; \,  X,Y,Z \in \mathcal{V}(M).
	\end{equation} 
\end{proposition}
\begin{proof}
	(Uniqueness). Suppose there exists such a covariant derivative. If for all $ X,Y $ and $ Z $ 
	we compute $ \nabla^n_Z \ll X,Y\gg_n $, $ \nabla^n_{ X} \ll Y,Z \gg_n $ and $ \nabla^n_Y \ll Z,X \gg_n$ by~\eqref{fc},
	then by subtracting the sum of the first two from the last one and applying the torsion-free property of a covariant derivative
	we obtain
	\begin{align} \label{hard}
	K_n(X,Y,Z)=& \mathcal{L}_{Z} \ll X,Y \gg_n + \mathcal{L}_X \ll Y,Z \gg_n - \mathcal{L}_Y \ll Z,X \gg_n  \nonumber \\ 
	-& \ll X, [Y,Z] \gg_n + \ll Y,[Z,X] \gg_n + \ll Z,[X,Y] \gg_n \nonumber \\ 
	=&2\ll \nabla^n_{ X}Y,Z \gg_n. 
	\end{align}
	Let $ \widehat{\nabla}^n $ be the other covariant derivatives satisfying \eqref{fc}.  The right-hand side of~\eqref{hard} does not
	depend on the covariant derivatives, therefore, for all $ n \in \nn $ we have $$ \ll \widehat{\nabla}^n_{ X}Y-\nabla^n_{ X}Y,Z \gg_n =0 .$$ Since $ Z $ is arbitrary,  the  Hausdorffness implies that $$ \widehat{\nabla}^n_{ X}Y = \nabla^n_{ X}Y .$$ 
	
	(Existence). Fix $ X,Y $, the function $ K_n(X,Y,Z) $ is smooth since it is the sum of smooth functions. The mapping $ K_n(X,Y,Z) \mapsto \mathcal{L}_ZK_n(X,Y,Z)$ is a continuous derivation so by Theorem~\ref{imp} for each $ n $
	there is a uniquely defined vector field which we call $\nabla^n_XY  $ such that $$ \ll \nabla^n_{ X}Y,Z \gg_n = \dfrac{1}{2} K_n(X,Y,Z) .$$
	Showing that $ \nabla^n_XY $ satisfies the properties $ (1)-(3)$ in Definition~\ref{cov} is standard.
	Therefore, it is omitted.  
\end{proof}
The preceding theorem and the ones we shall prove strongly depend on the nuclearness property of manifolds and the $\mc$-differentiability. They are not true for Fr\'{e}chet manifolds even for Banach manifolds with weak Riemannian metrics in general.

Henceforth, we assume that $M$ is a connected nuclear Fr\'{e}chet manifold of class $MC^{\infty}$ with a Riemann-Finsler structure 
$ ( \Finsler )_{n \in \nn} $.
Let $ x \in M $ and let $ B(0_x,r) $ be the open ball in $ T_xM $ centered at $ 0_x $ with  radius $ r $ with respect to the Finsler metric $\rho$~\eqref{finmetric}.
The injectivity radius of $ M $ at $ x $, $ i(x) $, is the least upper bound of  numbers $ r>0 $, such that
$ \exp_x $ is a diffeomorphism on  $ B(0_x,r) $.
\begin{theorem}\label{23}
	Let $ x \in M$, and let $ \varepsilon > 0 $ be such that $ \mathcal{U} = \exp_x (B (0_x, \varepsilon)) $ is a normal neighborhood 
	of $ x $. Then for any $ y \in \mathcal{U} $ there exists a unique geodesic $ \ell : [0,1] \to M $ joining $ x  $ and $ y $ such that
	for all $ n \in \nn $
	\begin{equation*}
	L_n(\ell) \leqq  \varepsilon. 
	\end{equation*}
\end{theorem}
\begin{proof}
	Let $ x \in M $ and let $ 0_x \in T_xM $ be the zero vector. On an open neighborhood $ N $ of $ 0_x $ in $ T_xM $ define the mapping 
	$ \varphi(v) = (x, \exp_{x}(v) )$. By virtue of Proposition~\ref{expo} in local charts, the Jacobin matrix of $ \varphi $
	at $ 0_x $ is
	$$ \varphi_{*}=
	\begin{bmatrix}
	id & 0 \\
	*  & id
	\end{bmatrix},
	$$
	which is invertible. Thus, by the inverse function theorem~\eqref{inv} $ \varphi $ is a diffeomorphism from some
	neighborhood $ W $ of $ 0_x $ onto its image. We can shrink $ W $ and assume that 
	$ W = \bigcup _{p \in V}B(0_p,\varepsilon) $ for some open neighborhood 
	$ V $ of $ x $. Then, for $ y \in \mathcal{U} $ there exists a unique $ v \in W $
	such that $ \varphi(v) = (x,y) $. That is, there exists a unique $ v \in B(0_x,\varepsilon) $
	such that $ \exp_x v = y$. Now define $ \ell(t): [0,1] \to M $ by $ \ell(t) = \exp_{x}(tv) $, this is a geodesic connecting $ x $ to $ y $
	and $ \ell'(0) =v $ and entirely is contained in $ \mathcal{U} $, since $B(0_x,\varepsilon)  $ is star-shaped and so 
	$ tv \in B(0_x,\varepsilon)  $ for $ t \in [0,1] $. Since $ \ell $ is contained in $ \mathcal{U} $ then for all $ n \in \nn$ we have $$ \sn \ell'(t) \sn_{\ell(t)}^n \leqq \varepsilon, $$ and so $ L_n (\ell) \leq \varepsilon $. 
	
	To prove the uniqueness let $ \alpha $ be another geodesic in $ \mathcal{U} $ connecting $ x,y $. We may assume that  $ \alpha(0)=1 $
	and $ \alpha(1)=y $ after an appropriate reparameterization. Then by Proposition~\ref{geod} we have
	$ \alpha(t) = \exp (t \alpha'(0)) $ for all $ t \in [0,1] $. Let $$ I =  \exp_x ^{-1}(\Img(\alpha)) .$$
	It is a line segment contained in $ B(0_x,\varepsilon) $ and its endpoints are $ 0_x $ and $ a \alpha'(0) $  for some $ a>0 $,
	because $ \Img (\alpha) \subset \mathcal{U} $ and the map $ \exp_x $ is a diffeomorphism so $ I $ is a connect closed
	subset in $$ \mathbb{A} = \big\{t\alpha'(0) \in T_xM \mid t \in (0,\infty)\big\}. $$ Now, we show that $ a\geq 1 $.
	If $ a < 1 $, then, the openness of $\mathcal{U}  $ yields there exists $  b \in (0,1] $ such that $ b \alpha'(0) \in \mathcal{U} $.
	But $$ \exp_x (b \alpha'(0)) \notin \Img (\alpha), $$ since $ \exp_x $
	is bijective on $ \mathbb{A} \bigcap \mathcal{U} $ and $ \exp_x (I) = \Img (\alpha) $.
	This is a contradiction because the image of the line segment connecting $ 0_x $
	and $ \alpha'(0) $ under $ \exp_x $ is $ \Img (\alpha) $.
	Thus, $ \alpha \geq 1 $  and so $ \alpha'(0) \in \mathcal{U} $. Therefore, $ \exp_x(\alpha'(0)) = \alpha(1) = y $
	and $ \alpha'(0) = \exp_x^{-1}(y) = v $, whence $ \alpha = \ell $.

\end{proof}   
Let $ I_1,I_2 $ be open intervals in $ \rr $ and let $ \ell: I_1 \times I_2 \to M,\, (t,s) \mapsto \ell(t,s)$ be an $MC^{\infty}$-curve.
Let $ \partial_i \ell $, $i=1,2 $, denote the ordinary partial derivative with respect to the $ i $-th variable. 
Since the curves $ t \mapsto \partial_i \ell $ and $ s \mapsto \partial_i \ell $ are lifts in $TM $ 
we can consider their covariant derivatives. 

For each $ n \in \nn $, let  $ \nabla^n_1 \partial_2 \ell$ be the covariant derivative of $\partial_2 \ell$ along
the curve $ \ell_s(t) = \ell(t,s) $ for a fixed $ s $. Similarly, let $ \nabla^n_2 \partial_1 \ell$ be the covariant derivative of $\partial_1 \ell$ along the curve $ \ell_t(s) = \ell(t,s) $ for each fixed $ t $. By Formula~\eqref{cdc} in a local chart $ U$
\begin{equation*}
\nabla^n_1 \partial_2 \ell_U = \partial_1\partial_2 \ell_U - \mycmd{S}_U(\ell_U)(\partial_1 \ell_U, \partial_2 \ell_U), 
\end{equation*}
and symmetry of $ \mycmd{S}_U $ implies that 
\begin{equation}\label{231}
\nabla^n_1 \partial_2 \ell = \nabla^n_2 \partial_1 \ell.
\end{equation}
therefore, for all $ n \in \nn $
\begin{equation}
\partial_2 \ll \partial_1 \ell, \partial_1 \ell  \gg_n = 2 \ll \nabla^n_2 \partial_1 \ell,\partial_1 \ell  \gg_n,
\end{equation}
so \eqref{231} follows that 
\begin{equation}\label{38}
\partial_2 \ll \partial_1 \ell, \partial_1 \ell  \gg_n = 2 \ll \nabla^n_1 \partial_2 \ell,\partial_1 \ell  \gg_n.
\end{equation}
Let $ \varepsilon>0 $ and $ x \in M $. Define a set $ S_{x,\varepsilon} \coloneq \{ v \in T_xM \mid \ll v, v \gg_n = \varepsilon^2 (\forall n \in \nn) \} $.

The following result generalizes the classical Gauss's lemma to the context of infinite dimensional $MC^{\infty}$-nuclear Fr\'{e}chet manifolds equipped with  Riemann-Finsler  structures. 
\begin{theorem}[Gauss's lemma]\label{gl}
	Let $ x_0 \in M $ and let $ (\mathcal{U}, \mathcal{W}) $ be a normal neighborhood of $ x_0 $.
	Then the geodesics through $ x \in \mathcal{U} $ are F-orthogonal to the image of $ S_{x,\varepsilon} $ under $ \exp_x $,
	for small enough $ \varepsilon > 0 $.
\end{theorem}
\begin{proof}
	For $ \varepsilon > 0 $ small enough, the map  $ \exp_x $ is defined on an open ball in $ T_xM $ of  radius slightly larger than $ \varepsilon $.
	The proof is equivalent to prove that for any $MC^{\infty}$-curve $ \jmath : I \to S_{x,1} $,  and  $ 0 < s < \varepsilon $, if we define
	\begin{equation*}
	\imath(s,t) = \exp (s \jmath(t))
	\end{equation*}
	then for any arbitrary $s_0,t_0$  the following curves
	\begin{equation*}
	t \to \exp_x(s_0\jmath(t)), \quad s \to \exp_x(s \jmath (t_0))
	\end{equation*}
	are F-orthogonal. 
	By proposition~\ref{geod} for each $ t $, the map $ \imath_t : s \to \imath(s,t) $ is a geodesic so for all $n \in \nn$
	$$ \nabla_1^n \partial_1 \imath = 0, $$ and $$ \partial_1 \ll \partial_1 \imath, \partial_1 \imath  \gg_n = 2 \ll \nabla^n_1 \partial_1 \imath, \partial_1 \imath  \gg_n =0,\, \forall n \in \nn.$$
	Thus, the functions 
	\begin{equation}\label{fin}
	s \mapsto \ll \partial_1 \imath(s,t), \partial_1 \imath(s,t)\gg_n
	\end{equation}
	    are constant for each $ t $.
	Since $ \partial_1 \imath(0,t) = \jmath (t) $  and $ \ll \jmath(t),\jmath(t) \gg_n=1 (\forall n \in \nn) $ it follows that
	$$ \ll \partial_1 \imath, \partial_1 \imath \gg_n=1 (\forall n \in \nn). $$ Therefore, 
	by~\eqref{38}
	\begin{equation*}
	\partial_1\ll \partial_1 \imath, \partial_2 \imath \gg_n = \ll \nabla_1 \partial_1 \imath, \partial_2 \imath  \gg_n 
	+ \dfrac{1}{2} \partial_2 \ll \partial_1 \imath, \partial_1 \imath   \gg_n = 0, \quad \forall n \in \nn.
	\end{equation*}
	Thereby, the functions $ s \mapsto \ll \partial_1 \imath(s,t), \partial_2 \imath(s,t)\gg_n  $ are constant 
	for each fixed $ t $. Let $ s=0 $, then $ \imath(0,t)= \exp_x(0)=x $ and therefore $ \partial_2 \imath (0,t)=0 $ for all $ t $.
	Thus, $$ \ll\partial_1 \imath, \partial_2 \imath  \gg_n =0 \, (\forall n \in \nn), $$ that is $ \partial_1 \imath $
	and $ \partial_2 \imath $ are  F-orthogonal. This concludes the proof.
\end{proof} 
\begin{theorem}\label{lenm}
	Let $ x \in M $ and $ \mathcal{U} = \exp_x (B (0_x, i(x)) $ be a normal neighborhood of $x$. Let $ \ell : [0,1] \to M $
	be the unique geodesic in $ \mathcal{U} $ joining $ x $ to  $ y \in \mathcal{U}$. Then, for any other piecewise $MC^1$- 
	path $ \imath : [0,1] \to M$ joining $ x,y $, we have 
	\begin{equation*}
	L_n(\ell) \leqq L_n(\imath), \quad \forall n \in \nn.
	\end{equation*}
	If the equality holds, then $ \jmath $ must coincide with  $ \ell $, up to reparametrization.
\end{theorem}
\begin{proof}
	Consider an $MC^1$-path $\imath : [0,1] \to \mathcal{U} $ connecting $ x $ to $ y $.
	Since $ \exp_x $ on $ B(0_x, i(x)) $ is a diffeomorphism we may find a unique  curve $$ t \to v(t): [0,1] \to T_xM$$
	with $ \sn v(t) \sn_{v(t)}^n = 1 \,(\forall n \in \nn) $ and a  curve $ r(t) : (0,1] \to (0,i(x)) $
	such that 
	\begin{equation*}
	\imath (t) = \exp_x (r(t)v(t))\coloneqq \mathbf{k}(r(t),t).
	\end{equation*}
	Locally, $ r(t)$ and $v(t) $ are obtained by the inverse of the exponential map after a smooth projection so $ r(t) $ and $ v(t) $ are piecewise $ MC^1 $.
	We may assume $ r(t) \neq 0 $, that is $ \imath(t) \neq x $ for all $ t \in (0,1] $ since otherwise
	we may define $ t_0 $ to be the last value such that $ \imath(t_0) = x $ and exchange $ \ell $ with $\imath \mid_{[t_0,1]}  $.
	Now we have 
	\begin{equation}\label{el}
	\imath'(t) = \partial_1 \mathbf{k}(r(t),t)r'(t) + \partial_2 \mathbf{k}(r(t),t).
	\end{equation}
	Also,
	\begin{equation*}
	\partial_1 \mathbf{k} = \big(T_{r v(t)} \exp_x \big)(v(t)) \quad \textrm{and} \quad 
	\partial_2 \mathbf{k} = \big(T_{r v(t)} \exp_x \big)(rv'(t)). 
	\end{equation*}
	By Theorem~\ref{gl}, $\partial_1 \mathbf{k}$ and $\partial_2 \mathbf{k}$ are F-orthogonal.
	By the same arguments for  proving~\eqref{fin} we have $$ \sn \partial_1 \mathbf{k} \sn^n_{v(t)}  = 1\, (\forall n \in \nn), $$ 
	and by~\eqref{el} we  obtain
	\begin{equation*}
	\big( \sn  \imath'(t) \sn_{\imath(t)}^n \big)^2 = \mid r'(t) \mid^2 + \big( \sn \dfrac{\dd \mathbf{k}(t)}{\dd t} \sn_{\mathbf{k}(t)}^n \big)^2 \geqq r'(t)^2.
	\end{equation*}
	Therefore, 
	\begin{equation}\label{last}
	L_n (\imath ) \geq \int_{0}^{1} \sn \imath(t) \sn_{\imath(t)}^n dt \geq \int_{0}^{1} \mid r'(t) \mid dt \geq r(1)- (\lim_{\epsilon \to 0}r(\epsilon)=\delta).
	\end{equation}
	Let $ y = \exp_x(rv) $ such that $ 0 < r < i(x) $ with $ v \in T_xM $ and $ \sn v \sn_x^n = 1 (\forall n \in \nn) $.
	
	For $ s, 0 < s <r $, the path $ \imath(t) $ contains a segment joining $ S_{x,s} $ and $ S_{x,r} $ and remains between them.
	By~\eqref{last} we have $ L_n(\imath) \geq r-\delta $ and so if $ \delta \to 0 $ then $ L_n(\imath) \geq r $. 
   Theorem~\ref{23} implies that there exists $ r_0 < i(x) $ such that $ L_{n}(\alpha) \leq r_0$ (we may find   $ u \in T_xM $ such that $ y = \exp(r_0u) $) but $  L_n (\imath)  \geq r_0 $, therefore for all $ n $ $$L_{n}(\alpha) \leq L_n(\imath)  .$$ 
	
	If $ L_{n}(\alpha) = L_n(\imath) $
	then in~\eqref{last} we must have the equality as well and this happens if and only if $ t \to v(t) $
	is constant and $ t \to r(t) $ is monotone. Thus, by a suitable reparametrization $ \imath $ becomes a geodesic. 
	Suppose this is the case, so $\imath:[0,r] \to M$ is the curve $ t \to  \exp_x(tv_0) $ and $ \exp_x(rv_0) =y $ for some $ v_0 \in T_xM $ with $ \sn v_0 \sn_x^n = 1 (\forall n \in \nn)$,
	but $ \exp_x $ is a diffeomorphism so $ v=v_0 $ and therefore $ \alpha = \imath $.
\end{proof} 

Let $ M $ be an $ MC^{\infty} $-nuclear Fr\'{e}chet manifold modeled on $ F $ with a Riemann-Finsler structure  $(\Finsler)_{n \in \nn}$.
Let a curve $ \ell:[a,b] \to M $ be an $MC^{\infty}$-curve. We denote the local representatives of $ \ell $ again by $ \ell $.  In a local chart $ U $, the coordinate of its
canonical lift is $ (\ell(t), \ell'(t)) $. For each $ n \in \nn $ we define the energy functional $ E_n $ by
\begin{equation*}
E_n(\ell) = \dfrac{1}{2} \int_a^b \ll \ell(t),\ell'(t)\gg_n  \dd t.
\end{equation*}
Take an $MC^{\infty}$-proper variation $ \mathbb{H} : (-\varepsilon, \varepsilon)\times [a,b] \to M $ of $ \ell $ such that 
$$ \mathbb{H}(0,s) = \ell(s), \quad \mathbb{H}(t,a) =\ell(a), \quad  \mathbb{H}(t,b) =\ell(b), $$
 for all $ t \in (-\varepsilon,\varepsilon) $.
 
Let $ \mathbb{H}_t(s)=\mathbb{H}(t,s)$, a curve $ \ell $ is called a critical point for  $ E_n $ if 
\begin{equation*}
\dfrac{\dd}{\dd t} \big( E_n (\mathbb{H}_t )\big) \mid_{t=0} = 0, \quad \forall n \in \nn.
\end{equation*}
The partial derivative of local representative of $ E_n : U \times F \to \rr $ are 
$$
\dd_1 E_n(u,e)(f) = \lim_{h \to 0} \dfrac{1}{h} \big(E_n(u+hf,e)-E_n(u,e) \big),
$$
$$
\dd_2 E_n(u,e)(f) = \lim_{h \to 0} \dfrac{1}{h} \big(E_n(u,e+hf)-E_n(u,e) \big).
$$
We will need the following result.
\begin{theorem}\cite[Theorem 6.3]{jo} \label{jos}
	Let $\mathscr{L} \in C^{\infty}(TM,\rr)$ be a Lagrangian. Then a smooth curve $J(t)$ is critical for $\mathscr{L}$
	if and only if it satisfies the Euler-Lagrange equation
		\begin{equation}
	(\dd_1 {L} )(\jmath(t),\jmath'(t)) - \dfrac{\dd}{\dd h}\mid_{h=t} (\dd_2 {L} )(\jmath(h),\jmath'(h)) = 0,
	\end{equation}
	in a local chart where $L$ and $\jmath(t)$ are, respectively, the local expressions of $L$ and $J(t)$, and
	$\dd_i L (i \in {1, 2})$ are the partial derivatives of $L$.
\end{theorem}
We should mention that in the preceding theorem the used  differentiability is equivalent to the Keller's differentiability,
as we have seen functions  $\ll \cdot, \cdot \gg_n$ are   Keller's differentiable so we can apply it.
\begin{theorem}\label{elq}
	An $MC^{\infty}$-curve $ \ell : [a,b] \to M  $ is geodesic if and and only if in a local chart it satisfies the Euler-Lagrange equations
	\begin{equation}\label{lj}
	(\dd_1 E_n)(\ell(t),\ell'(t)) - \dfrac{d}{dh}\mid_{h=t} (\dd_2 E_n )(\ell(h),\ell'(h)) = 0, \forall n \in \nn.
	\end{equation}
\end{theorem}
\begin{proof}
	For an $MC^{\infty}$-variation $ \mathbb{H}: (t,s) \mapsto \mathbb{H}(t,s) $, along $ \mathbb{H} $ define the vector fields 
	$$Y \coloneq \dd \mathbb{H} (\partial / \partial t), \quad X \coloneq \dd \mathbb{H} (\partial / \partial s). $$ 
	For all $n \in \nn$ we have
	\begin{align*}
	\dfrac{\dd}{\dd t} (E_n (\mathbb{H}_t))  
	=& \dfrac{1}{2}  (\int_{a}^{b} \dfrac{\dd}{\dd t} \ll X, X \gg_n)  ds \nonumber \\ 
	=& \int_{a}^{b} \ll \nabla^n_{Y}X , X \gg_n ds \quad \textrm{ since $\nabla^n$ is compatible}\nonumber \\ 
	=& \int_{a}^{b} \ll \nabla^n_{X}Y, X \gg_n ds \quad \textrm{since  $\nabla^n$ is torsion-free}\nonumber \\  
	=& \int_{a}^{b} \big(\dfrac{\dd}{\dd s} \ll Y, X \gg_n - \ll   Y,\nabla^n_{X}  X \gg_n \big)  ds \nonumber \\
	=&  \ll Y, X \gg_n \mid_a^b- \int_{a}^{b} \ll   Y,\nabla^n_{X}  X\gg_n  ds \nonumber \\ 
	\end{align*}
	Since the variation is proper we have $$ Y(a) = Y(b)= 0 .$$ Moreover, $  X(0,s) = \partial \mathbb{H} / \partial s (0,s) = \ell'(s) $, therefore,
	\begin{equation*}
	\dfrac{\dd}{\dd t}(E (\mathbb{H}_t) \mid_{t =0} = - \int_{a}^{b} \ll Y(0,s), (\nabla^n_{\ell'}\ell')(s) \gg_n ds. 
	\end{equation*}
	The right side is zero if and only if $ \ell $ is geodesic. That is, the critical points are geodesic and hence by Theorem~\ref{jos}
	they need to satisfy the Euler-Lagrange equations \eqref{lj}.
\end{proof}
Let $N$ be a closed  Einstein manifold of dimension $ n $. The manifold of Riemannian metrics on $ N $, $ \mathcal{M} $, is a nuclear Fr\'{e}chet manifold, it is also $MC^{\infty}$ (see~\cite{k5,ml}).
The solution to the Ricci flow equation
$$
\dfrac{\dd g(t)}{\dd t} = -2 \textrm{Ric}(g(t))
$$ 
is $ g(t) = (1-2\lambda)tg_0 $, where $ g_0 $ is a Riemannian metric and $ \textrm{Ric}(g_0) = \lambda g_0$, see~\cite{ha}. This is a curve on $ \mathcal{M} $. In local charts, obviously $ g(t) $ is $ C^1 $ and 
$$ g'(t)=-2\lambda g_0 \in \mathbb{L}_{\sigma,\varrho}([0,T],F),$$ where $\sigma$ is the standard metric on $\rr$, $ T$ is a time less than the finite singular time and $$  g': [0,T] \to \mathbb{L}_{\sigma,\varrho}([0,T],F)$$ is constant  and hence a continuous map into $\mathbb{L}_{\sigma,\varrho}([0,T],F) $.
Thus, $ g(t) $ is $ MC^1 $ and by induction it follows that $ g(t) $ is $ MC^k $ with $$ g^{(k)} =0, (k\geq 2).$$

Simple calculations show that for all $ n $ we have $$(\dd_1 E_n)(g(t),g'(t)) \neq 0, \quad \dfrac{\dd}{\dd h}\mid_{h=t} (\dd_2 E_n )(g(h),g'(h)) = 0.$$
So, the Euler-Lagrange equations do not hold, therefore, $ g(t) $ is not geodesic. This result is proved in~\cite{rz1} by using the geodesic equation
on the manifold of Riemannian metrics which is considered as the projective limit of Banach manifolds.

\end{document}